\documentclass[12pt, letterpaper]{article}
\usepackage{amssymb}
\usepackage{amsfonts}
\usepackage{amsmath,amsthm}
\usepackage[unicode]{hyperref}
\usepackage{tikz}

\DeclareMathOperator{\ad}{ad}
\DeclareMathOperator{\Tr}{Tr}

\newtheorem{thm}{Theorem}[section]
\newtheorem{prop}[thm]{Proposition}
\newtheorem{lem}[thm]{Lemma}
\newtheorem{Cor}[thm]{Corollary}

\newtheorem{conj}[thm]{Conjecture}

\makeindex

\title{Lie algebra of  column-finite infinite matrices: ideals and derivations}

\author{Waldemar Ho{\l}ubowski\footnote{Corresponding author, Email address: w.holubowski@polsl.pl}  and Sebastian \.Zurek}

\date{}

\begin{document}

\maketitle

Silesian University of Technology, Faculty of Applied Mathematics\\
Kaszubska 23, 44-101 Gliwice, Poland

\begin{abstract}
In this paper we shall consider the Lie algebra of column-finite infinite matrices indexed by positive integers $\mathbb{N}$,  describe the lattice of its ideals for arbitrary field $K$  and study its derivations over any commutative, unital ring $R$.
\end{abstract}



\section{Introduction}
The works of S.~Lie, W.~Killing and E. Cartan were the starting points for systematic development of the theory of  finite-dimensional Lie algebras. We mention here the classification of finite dimensional simple Lie algebras over the algebraically closed fields (for fields of characteristic $0$ due to E.~Cartan and W.~Killing and for characteristic $p>3$ given in works of R.~E.~Block, R.~L.~Wilson,  H.~Strade, A.~Premet) and the representation theory (the highest weight classification of irreducible modules of general linear Lie algebras). The nice exposition can be found in \cite{Bou1},  \cite{Bou2},  \cite{Bou3}, \cite{Jac62}, \cite{Hum72}, \cite{Str13}.

However at the present time there is no general theory of the infinite-dimensional Lie algebras, they have been investigated only in special settings. There are few classes of   infinite-dimensional Lie algebras that were
more or less intensively studied from the geometric point of view.
In preface to his monography \cite{Kac90}  V.~Kac consider as the first class the Lie algebras of vector fields which have many geometric applications (cohomology theory  \cite{FT83}, \cite{Fuks86}.)
The second class  consists of Lie algebras of smooth mappings of a given manifold into a finite-dimensional Lie algebra (some central extensions of these algebras are referred as current algebras).
The third class consists of the classical Lie algebras of operators in a Hillbert or Banach space \cite{delaHar72}.
Nowadays, there are studied intensively the  Kac-Moody algebras defined by generalized Cartan matrices (the fourth class), which have many applications in modern physics \cite{Kac90}, \cite{Bau90}, \cite{Bau97}.

Algebraic point of view was used in investigations of free Lie algebras and graded Lie algebras. In the book by Amayo and Stewart \cite{AS74} the results on the lattice structure of infinite-dimensional Lie algebras are surveyed. They translate some group theoretic definitions and results on infinite discrete groups into the language of Lie algebras. They provide results on soluble ideals, locally nilpotent radicals, Lie algebras with chain or other finitness restrictions, Frattini theory, Engel conditions, etc.. This shows the deep parallel of both theories however the proofs may differ, because Lie algebras depend on underlying field. J. E. Humphreys wrote: "group theorists will correspondingly find the book useful and attractive, though its appeal to the general reader will be more limited."  \cite{Hum74}.

 In many papers appear, as examples, the Lie algebras of infinite matrices.
 We mention here the Lie algebra $\mathfrak{g}\mathfrak{l}_{\infty}$  of $\mathbb{Z} \times \mathbb{Z}$ infinite matrices over $\mathbb{C}$ which have only finite number of nonzero entries:
 $$\mathfrak{g}\mathfrak{l}_{\infty}= \{ A=(a_{ij})_{i,j \in \mathbb{Z}} :  \,\,{\rm all \,\,but\,\, the\,\, finite\,\, number \,\,of}\,\, a_{ij} \,\,\ {\rm are} \,\,0\} $$

 \noindent It has a basis consisting of matrices $E_{ij}$ which have $1$ in $(i,j)$-entry and zeros elsewhere and so it is countably dimensional. Moreover, $\mathfrak{g}\mathfrak{l}_{\infty}$ is a subalgebra of $\mathfrak{g}\mathfrak{l}_{J}$ --  the Lie algebra of generalized Jacobian matrices, i.e. infinite matrices having nonzero entries in finite number of diagonals, introduced in \cite{KP81}, 
 \cite{DJKM81}:
 \begin{displaymath}
 \mathfrak{g}\mathfrak{l}_{J}= \{ A=(a_{ij})_{i,j \in \mathbb{Z}} :  \,\,\ a_{ij} = 0 \,\,\,\ {\rm for\,\, all} \,\,\,| i-j| >> 0 \}. 
\end{displaymath}

\noindent    The last algebra is denoted in \cite{KacB13} as $\bar{a}_{\infty}$ and is uncountably dimensional.   The representations of $\mathfrak{g}\mathfrak{l}_{\infty}$ play important role in soliton's theory . A central extension $a_{\infty}$ of $\bar{a}_{\infty}$ has also many applications in physics. It contains for example any affine Kac-Moody algebra and Virasoro algebra \cite{KacB13}.

We mention here also a work of A. A. Baranov describing simple finitary Lie algebras over an algebraically closed fields of characteristic $0$ \cite{Bar1}. This result was generalized to arbitrary fields of characteristic $0$ in \cite{Bar2} and to any algebraically closed field of characteristic $p > 3$ 	by A.A. Baranov and H. Strade in \cite{BarStr} (with use of classification of finite dimensional simple Lie algebras  over positive characteristic). A.A. Baranov and A. G. Zhilinskii   classified simple locally finite Lie algebras of diagonal type \cite{BarZhi}. The nice exposition of these results and results expoiting the connections with Jordan structures in Lie algebras and celebrated E.I Zelmanov result on local finitness of Lie PI-algebras one can find in monography by A. Lopez \cite{Lopez}.

Quite recently I. Penkov and V. Serganova in \cite{PenSer} introduced new Lie algebras motivated by representation theory and Mackey dissertation \cite{Mackey} on linear pairings. Let $p: V \times W \to K$ be a nondegenerate pairing for some vector spaces, not necessarily of countable dimension or having same dimension. The endomorphisms of the pairing $p$ are denoted $\mathfrak{g}\mathfrak{l}^{M}(V, W)$ and called Mackey Lie algebra of $p$. If $dim(V) = \aleph_0$ and we fix a basis $\{e_i\}_{i \in \mathbb{N}}$, then $gl_M(V, V^{*})$ is identified with the Lie algebra of column-finite infinite matrices $\mathfrak{g}\mathfrak{l}_{cf}(\mathbb{N},K) $. In the series of papers \cite{Chir}, \cite{ChPen1}, \cite{ChPen2} there were  studied representation theory of such algebras and quite general categories connected with these representations.

  However there is no systematic study of  Lie algebras of infinite matrices. In this paper we shall consider the Lie algebra of column-finite infinite matrices indexed by positive integers $\mathbb{N}$,  describe the lattice of its ideals for arbitrary field $K$  and study its derivations over any commutative, unital ring $R$.

\bigskip

Let $R$ be a commutative, unital ring and let $M_{cf}(\mathbb{N},R)$ denote the set of all $\mathbb{N} \times \mathbb{N}$  column-finite matrices over $R$, i.e. infinite matrices which have only finite number of nonzero entries in each column.
It forms an associative $R-$algebra. In case $R=K$  -- a field, the column-finite matrices correspond to the linear endomorphisms of the vector space $K^{(\mathbb{N})}$ with respect to the canonical basis
 (elements of  $K^{(\mathbb{N})}$ are infinite columns $(a_1, \ldots , a_n, \ldots)^T$ with only finite number of nonzero entries). It is clear that  $M_{cf}(\mathbb{N},K)$ is not countably dimensional as a vector space and it is imposssible to find reasonable basis for it.

 $M_{cf}(\mathbb{N},R)$ with respect to Lie product $[A,B]=AB-BA$ forms a Lie algebra  $\mathfrak{g}\mathfrak{l}_{cf}(\mathbb{N}, R)$. By  $\mathfrak{g}\mathfrak{l}_{rcf}(\mathbb{N}, R)$ we denote its Lie subalgebra 
of all matrices which are simultaneously row-finite and column-finite. The additional condition of row-finitness means that also the transpose of the  matrix defines an endomorphism of  $K^{(\mathbb{N})}$,
 i.e. the adjoint endomorphism of the dual space  $K^{\mathbb{N}}$ preserves the subspace  $K^{(\mathbb{N})}$   ($K^{\mathbb{N}}$ consists of all infinite columns  $(a_1, \ldots , a_n, \ldots)^T$). 

\smallskip

By  $\mathfrak{g}\mathfrak{l}_{fr}(\mathbb{N}, R)$ we denote the Lie subalgebra of  $\mathfrak{g}\mathfrak{l}_{cf}(\mathbb{N}, R)$ consisting of all matrices which have nonzero entries in only finite number of rows. It is clear that 
$$\mathfrak{g}\mathfrak{l}_{fr}(\mathbb{N}, R) = \cup_{n>0}\mathfrak{g}\mathfrak{l}(n, \mathbb{N}, R),$$
where  $\mathfrak{g}\mathfrak{l}(n,\mathbb{N}, R)$ is a Lie subalgebra of matrices which have nonzero entries in the first $n$ rows.
For any matrix $A$ from  $\mathfrak{g}\mathfrak{l}(n,\mathbb{N}, R)$ we can define a trace $\textrm{Tr}(A)$ as a sum of first $n$  entries on the main diagonal. In fact, we can extend $\textrm{Tr}(A)$ onto  $\mathfrak{g}\mathfrak{l}_{fr}(\mathbb{N}, R)$. One can check that $\textrm{Tr}$ is a Lie homomorphism 
$$\textrm{Tr}:  \mathfrak{g}\mathfrak{l}_{fr}(\mathbb{N}, R) \to R$$ 
onto a commutative Lie algebra $R$. Its kernel  $\mathfrak{s}\mathfrak{l}_{fr}(\mathbb{N}, R)$ is an ideal in  $\mathfrak{g}\mathfrak{l}_{fr}(\mathbb{N}, R)$ and consists of all matrices with trace $0$. 

By  $\mathfrak{d}_{sc}(\mathbb{N}, R)$ we denote the Lie subalgebra of scalar matrices $\alpha \cdot E$, where $\alpha \in R$ and
$E$ is the infinite identity matrix. Of course, $\mathfrak{d}_{sc}(\mathbb{N}, R)$ is contained in the center of 
 $\mathfrak{g}\mathfrak{l}_{cf}(\mathbb{N}, R)$.

The direct limit $\mathfrak{g}\mathfrak{l}_{\infty}(R)$ of Lie  algebras $\mathfrak{g}\mathfrak{l}_n(R)$ of $n \times n$ matrices under natural embeddings $\mathfrak{g}\mathfrak{l}_n(R) \to \mathfrak{g}\mathfrak{l}_{n+1}(R)$, given by:
$$A \rightarrow \left(\begin{array}{c|c} A&0\cr \cline{1-2}
0&0\cr
\end{array}\right)$$
is the Lie algebra of countable dimension. It can be viewed as  the Lie algebra of infinite $\mathbb{N} \times \mathbb{N}$ matrices $A$, which have only finite number of nonzero entries. Similarly, we obtain a Lie subalgebra  $\mathfrak{s}\mathfrak{l}_{\infty}(R)$ of matrices $A$ having $\Tr (A)=0$. By results of Ado and Iwasawa any   finite-dimensional Lie algebra over a field $K$  can be embedded into $\mathfrak{g}\mathfrak{l}_n(K)$  for some natural $n$ and so into $\mathfrak{g}\mathfrak{l}_{\infty}(K)$.

It is known that any Lie algebra which has a countable dimensional faithfull representation can be embedded into 
$\mathfrak{g}\mathfrak{l}_{cf}(\mathbb{N}, K)$. This applies in particular to any countable dimensional simple Lie algebra (as it adjoint representation is faithfull). Classifying  Lie algebras  with a faithfull countable dimensional representation is clearly not a reasonable problem.

\noindent K.R. Goodearl, P. Menal and J. Moncasi proved in \cite{GMM93} that any countably dimensional associative $K-$ algebra can be embedded into the 
algebra of row-finite and column-finite infinite $\mathbb{N} \times \mathbb{N}$  matrices. It means that any countably dimensional Lie algebra which arise from such algebra by taking a Lie product can be embedded into
 $\mathfrak{g}\mathfrak{l}_{rcf}(\mathbb{N}, K)$. This result is stronger than that with use of adjoint representation.

The simplicity of $\mathfrak{s}\mathfrak{l}_{fr}(\mathbb{N}, K)$ for any field $K$ was proved in \cite{Hol16}, \cite{HZ16}. 
In fact, $\mathfrak{s}\mathfrak{l}_{fr}(\mathbb{N}, K)$ is a matrix representation of finitary simple special linear group $\mathfrak{s}\mathfrak{l}(V)$ for countably dimensional vector space $V$ over $K$. In \cite{Bar1} simplicity of $\mathfrak{s}\mathfrak{l}(V)$
was proved for algebraically closed field of characteristic $0$, generalized to any field of characteristic $0$ in \cite{Bar2}.
For algebraically closed fields of characteristic $p>3$ it was proved in \cite{BarStr}. These proofs rely on results in representation theory and classification of finite dimensional simple Lie algebras over a field of positive characteristic. 
Our proofs of simplicity of  $\mathfrak{s}\mathfrak{l}_{fr}(\mathbb{N}, K)$ are short and elementary, uses only matrix computations and are valid for any field $K$.

The results above give motivation for a systematic study of $\mathfrak{g}\mathfrak{l}_{cf} (\mathbb{N}, R)$.

We will use the direct sum decomposition
$\mathfrak{g} \oplus \mathfrak{h} $ of two Lie algebras $\mathfrak{g}$, $\mathfrak{h} $ in usual sense \cite{Hen12}. Our main result is the following

\begin{thm} For any field $K$ the following subalgebras $\{0\}$, 
$\mathfrak{d}_{sc}(\mathbb{N}, K)$, 
  $\mathfrak{s}\mathfrak{l}_{fr}(\mathbb{N}, K)$,
 $\mathfrak{g}\mathfrak{l}_{fr}(\mathbb{N}, K)$,
 $\mathfrak{d}_{sc}(\mathbb{N}, K) \oplus\mathfrak{s}\mathfrak{l}_{fr}(\mathbb{N}, K)$,
 $\mathfrak{J}_{k}$ for $k \in K \setminus\{0\}$,
 $\mathfrak{d}_{sc}(\mathbb{N}, K)  \oplus\mathfrak{g}\mathfrak{l}_{fr}(\mathbb{N}, K)$,
$\mathfrak{g}\mathfrak{l}_{cf}(\mathbb{N}, K)$
are the only  ideals of  $\mathfrak{g}\mathfrak{l}_{cf}(\mathbb{N}, K)$.
\end{thm}

In \cite{PenSer}
the authors announced without the proof (at the end of chapter $6$) that 
$\mathfrak{g}\mathfrak{l}_{cf}(\mathbb{N}, \mathbb{C})$ has only five proper ideals. They ommited an uncountable family of ideals 
$\mathfrak{J}_{k}$ for $k \in \mathbb{C} \setminus\{0\}$. Our result fill this gap and give descriptions of all ideals for any field $K$.

The lattice of ideals of $\mathfrak{g}\mathfrak{l}_{cf}(\mathbb{N}, K)$  is shown in the  figure below (we abbreviate notation for more convenience, see the beginning of Section $2$ for definition of $\mathfrak{J}_{k}$.)

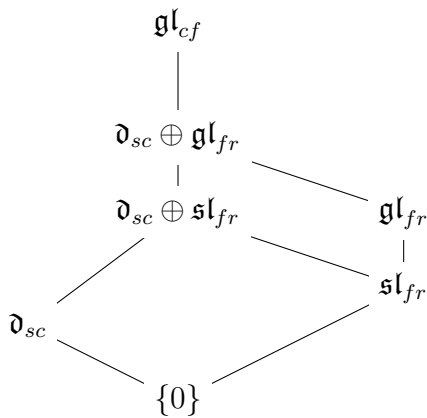
\begin{center}
\begin{tikzpicture}
 \draw (0,0.5)--(-2,1.5);
 \draw (-2,1.5)--(0,2.8);
 \draw (0,2.8)--(2,4);
 \draw (2,4)--(2,5.5);
 \draw (0,0.5)--(2,1.5);
 \draw (2,1.5)--(4,2.8);
 \draw (4,2.8)--(2,4);
 \draw (2,1.5)--(0,2.5);

 \draw (2,1.5)--(2,2.8);
 \draw (2,2.8)--(2,4);
	 \draw (2,1.5)--(1.5,2.6);
	 \draw (2,1.5)--(2.5,2.6);
	 \draw (1.5,2.6)--(2,4);
	 \draw (2.5,2.6)--(2,4);
 
 \fill (0,0.5) node [fill=white]{\{0\}};
 \fill (-2,1.5) node [fill=white]{$\mathfrak{d}_{sc}$};
 \fill (2,1.5) node [fill=white]{$\mathfrak{sl}_{fr}$};
 \fill (4,2.8) node [fill=white]{$\mathfrak{gl}_{fr}$};
 \fill (2,2.8) node [fill=white]{$\ldots$ $\mathfrak{J}_{k}$ $\ldots$};
 \fill (0,2.8) node [fill=white]{$\mathfrak{d}_{sc} \oplus \mathfrak{sl}_{fr} $};
 \fill (2,4) node [fill=white]{$\mathfrak{d}_{sc} \oplus \mathfrak{gl}_{fr} $};
 \fill (2,5.5) node [fill=white]{$\mathfrak{gl}_{cf}$};
\end{tikzpicture}
\end{center}

We note  that in finite-dimensional case the Lie algebra $\mathfrak{gl}(n, \mathbb{C})$ has only four ideals 
$\{0\}$, $\mathfrak{d}(n, \mathbb{C})$, $\mathfrak{sl}(n,\mathbb{C})$ and $\mathfrak{gl}(n, \mathbb{C})$ \cite{Hen12} and has the following lattice of ideals:

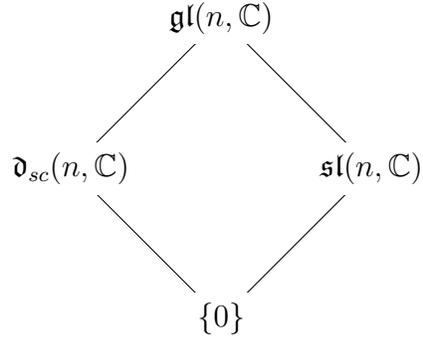
\begin{center}
\begin{tikzpicture}
 \draw (0,0)--(-2,2)--(0,4)--(2,2)--(0,0);
 \fill (0,0) node [fill=white]{$\{0\}$};
 \fill (-2,2) node [fill=white]{$\mathfrak{d}_{sc}(n, \mathbb{C})$};
 \fill (2,2) node [fill=white]{$\mathfrak{sl}(n, \mathbb{C})$};
 \fill (0,4) node [fill=white]{$\mathfrak{gl}(n, \mathbb{C})$}; 
\end{tikzpicture}
\end{center}

 As a consequence of Theorem 1.1.   we obtain

\begin{Cor}
The Lie algebra $\mathfrak{s}\mathfrak{l}_{fr}(\mathbb{N}, K)$ and the factor Lie algebra 
$\mathfrak{gl}_{cf}(\mathbb{N}, K)/(\mathfrak{d}_{sc}(\mathbb{N}, K) \oplus \mathfrak{gl}_{fr}(\mathbb{N}, K))$  are  simple.
\end{Cor}

Our proof of Theorem 1.1. gives a different proof of simplicity of $\mathfrak{s}\mathfrak{l}_{cf}(\mathbb{N}, K)$  than that given in \cite{Hol16}, \cite{HZ16}.
We note that the simplicity of   $\mathfrak{gl}_{cf}(\mathbb{N}, \mathbb{C})/(\mathfrak{d}_{sc}(\mathbb{N}, \mathbb{C}) \oplus \mathfrak{gl}_{fr}(\mathbb{N}, \mathbb{C}))$
was proven in in \cite{PenSer} (Theorem 6.3.b)). We generalize this result to any field $K$. The proof in \cite{PenSer} uses among others generalized eigenspaces, infinity of $\mathbb{C}$ and so it can not be generalized to any field $K$.

Direct consequence of the Theorem 1.1.  is the following
\begin{Cor}
The derived algebra of $\mathfrak{g}\mathfrak{l}_{cf}(\mathbb{N}, K)$  coincides with 
$\mathfrak{g}\mathfrak{l}_{cf}(\mathbb{N}, K)$, i.e. it is perfect. 
The ideal $\mathfrak{d}_{sc}(\mathbb{N}, K)$ is the center of $\mathfrak{g}\mathfrak{l}_{cf}(\mathbb{N}, K)$.
\end{Cor}
In view of above results we can make the following conjecture
\begin{conj}
Lie algebra of endomorphisms of any infinite dimensional vector space is perfect. The center of this
algebra coincide with the set of scalar operators.
\end{conj}

The second main result describes derivations.
\begin{thm}
Let $R$ be any commutative unital ring. Every derivation of $\mathfrak{g}\mathfrak{l}_{cf}(\mathbb{N}, R)$  is a sum of inner and central derivations.
\end{thm}
\begin{Cor}
For any field $K$ every derivation of $\mathfrak{gl}_{cf}(\mathbb{N}, K)$ is inner and
\begin{center}
$Der(\mathfrak{gl}_{cf}(\mathbb{N}, K) \cong \mathfrak{gl}_{cf}(\mathbb{N}, K)/ \mathfrak{d}_{sc}(\mathbb{N}, K)$
\end{center}
\end{Cor}

We refer to section 3 for definitions. The derivations of $\mathfrak{gl}(n, R)$ were described in \cite{WY06}.
In \cite{Noe05} it was shown, among others, that the Lie algebra of derivations of $\mathfrak{s}\mathfrak{l}_{\infty}(K)$ ($K$ a field of characteristic $0$) is isomorphic to
factor algebra $\mathfrak{g}\mathfrak{l}_{rcf}(\mathbb{N}, K)/ \mathfrak{d}_{sc}(\mathbb{N}, K)$. 
The description given in \cite{Noe05} complements decription from \cite{Str99} of derivations of locally finite split simple Lie algebras over a field of characteristic zero. It provides additional information that leads for the split case to the aforementioned description by infinite matrices.

In \cite{HKZ17} the derivations of the Lie algebra of strictly uppertriangular matrices 
 $\mathfrak{n}(\mathbb{N}, R)$ for any commutative unital ring $R$ were described (it seems to be  a first example of finding derivations for uncountably dimesional Lie algebra of infinite matrices, our Theorem $1.5$ gives a second example).

Our results are proved for $\mathfrak{g}\mathfrak{l}_{cf}(\mathbb{N}, R)$, but  the Proposition 3.4. says that $\mathfrak{g}\mathfrak{l}_{cf}(\mathbb{N}, R)$  and $\mathfrak{g}\mathfrak{l}_{cf}(\mathbb{Z}, R)$ - the Lie algebra of 
$\mathbb{Z} \times \mathbb{Z}$ column-finite infinite matrices are isomorphic. 
So $\mathfrak{gl}_{cf}(\mathbb{Z}, R)$ has the same lattice of ideals (if $R=K$) and similar description of derivations.
The proofs for $\mathbb{N} \times  \mathbb{N}$ matrices  are technically simpler.

We note that $\mathfrak{g}\mathfrak{l}_{cf}(\mathbb{Z}, \mathbb{C})$ contains a Lie algebra of generalized Jacobian matrices $\mathfrak{gl}_J$.
Recently, A. Fialowski and K. Iohara in \cite{FI17}, \cite{FI18} studied homology of $\mathfrak{gl}_J$ and its subalgebras. They remarked that there is a little known about algebraic properties of $\mathfrak{gl}_J$. The authors plan to fill this gap in forthcoming paper. The following problem seems to be interesting:
\begin{quote}
Describe the lattice of ideals of Lie algebra of generalized Jacobian matrices.
\end{quote}

In section $2$ we give the proofs of Theorem 1.1. and corollaries, in section $3$ we prove Theorem 1.5. and Proposition 3.4.


\section{Ideals of the Lie algebra of column-finite infinite matrices}

In this section we describe ideals of  the Lie algebra $\mathfrak{gl}_{cf}(\mathbb{N},K)$ over a field $K$.
We recall that a subspace $\mathfrak{I}$ is an ideal of Lie algebra $\mathfrak{L}$, if it  satisfies the condition $[\mathfrak{L},\mathfrak{I}]\subseteq \mathfrak{I}$, we denote it $\mathfrak{I}\triangleleft \mathfrak{L}$. This condition is equivalent to $[\mathfrak{I},\mathfrak{L}]\subseteq \mathfrak{I}$ (it follows from anticommutativity and bilinearity of Lie product).

By $\langle \mathfrak{g}_{1}, \mathfrak{g}_{2} \rangle$ we denote the interval of ideals of $\mathfrak{gl}_{cf}(\mathbb{N},\mathbb{K})$ between 
 $\mathfrak{g}_{1}$ and  $\mathfrak{g}_{2}$, i.e.the set of  all $\mathfrak{I}\triangleleft \mathfrak{gl}_{cf}(\mathbb{N},\mathbb{K})$ such that
 $\mathfrak{g}_{1} \subseteq \mathfrak{I} \subseteq \mathfrak{g}_{2}$.  The ideal $\mathfrak{I}$ is called proper if $\mathfrak{I} \not=\mathfrak{g}_{1}$ and  $\mathfrak{I} \not=\mathfrak{g}_{2}$.

Since in this section $K$ is the fixed field, we will use abbreviation  $\mathfrak{gl}_{cf}$ for  $\mathfrak{gl}_{cf}(\mathbb{N},K)$. Similarly we use abbreviations 
$\mathfrak{d}_{sc}$, 
  $\mathfrak{s}\mathfrak{l}_{fr}$,
 $\mathfrak{g}\mathfrak{l}_{fr}$,
 $\mathfrak{d}_{sc} \oplus\mathfrak{s}\mathfrak{l}_{fr}$,
 $\mathfrak{d}_{sc} \oplus\mathfrak{g}\mathfrak{l}_{fr}$ for its corresponding  subalgebras. 
 By $\{0\}$ we denote the  trivial algebra. 

We denote by  $E$ -- the infinite identity matrix and by $E_{ij}$ - the matrix  which has $1$ in $(i,j)$-entry and zeros elsewhere.
For a matrix $A=(a_{ij}) \in \mathfrak{gl}_{cf}$ we will write formally $A= \sum _{ij}a_{ij} E_{ij}$.

Let $\mathfrak{J}_{k}=\{  m(E+k E_{11}) | m \in K \}  \oplus \mathfrak{sl}_{fr}$.


\begin{lem}\label{idCF}
Lie subalgebras 
$\{0\}$, 
$\mathfrak{d}_{sc}$, 
  $\mathfrak{s}\mathfrak{l}_{fr}$,
 $\mathfrak{gl}_{fr}$,
 $\mathfrak{d}_{sc} \oplus\mathfrak{s}\mathfrak{l}_{fr}$,
 $\mathfrak{J}_{k}$ for $k \in K$,
 $\mathfrak{d}_{sc} \oplus\mathfrak{g}\mathfrak{l}_{fr}$ and  $\mathfrak{gl}_{cf}$
 are ideals in the Lie algebra $\mathfrak{gl}_{cf}$.
\end{lem}
\begin{proof}
It is clear  that $\{0\}\triangleleft\mathfrak{gl}_{cf}$ and $\mathfrak{gl}_{cf}\triangleleft\mathfrak{gl}_{cf}$. 
Any  $A\in\mathfrak{d}_{sc}$ commutes with every $B\in\mathfrak{gl}_{cf}$, which means that  $[A,B]=0$ and  so $\mathfrak{d}_{sc}\triangleleft\mathfrak{gl}_{cf}$.

Let $A\in\mathfrak{gl}_{fr}$ and $n$ be a number of the last nonzero row in the matrix $A$. For all $B\in\mathfrak{gl}_{cf}$ the  matrix $AB$ has at most $n$ nonzero rows. 
Since  $B$ is column-finite, it  has only finite numbers of nonzero entries in the first $n$ columns so $BA\in\mathfrak{gl}_{fr}$. Finally
$[A,B]=AB-BA\in\mathfrak{gl}_{fr}$.

Let $A\in\mathfrak{sl}_{fr}$ and $B\in\mathfrak{gl}_{cf}$. We know that $AB\in\mathfrak{gl}_{fr}$ and $BA\in\mathfrak{gl}_{fr}$. Let $m$ be the number of the last
nonzero row in the matrix $AB$, $k$ be the number of the last nonzero row in the matrix $BA$ and $n=\max\{m,k\}$. Denote by $(AB)_n$ and $(BA)_n$ square  $n\times n$ matrices 
from upper left corner of  $AB$ and $BA$ respectively. We have $\textrm{Tr }(AB-BA)=\textrm{Tr }((AB)_n-(BA)_n)=0$.
From the bilinearity of Lie bracket and previous considerations it follows that $\mathfrak{J}_{k}$ is also an ideal of $\mathfrak{gl}_{cf}$
\end{proof}

\begin{lem}\label{sl}
Let  $\mathfrak{I}\triangleleft \mathfrak{gl}_{cf}$. Then either $\mathfrak{I}\subset \mathfrak{d}_{sc}$ or $\mathfrak{sl}_{fr} \subset \mathfrak{I}$.
\end{lem}
\begin{proof}
Step 1: We show that if $A$ is diagonal matrix and $a_{ii}\not=a_{jj}$ then $E_{ij}\in\mathfrak{I}$.

\noindent We have $E_{ij}=(a_{ii}-a_{jj})(a_{ii}-a_{jj})^{-1}E_{ij}=[A,(a_{ii}-a_{jj})^{-1}E_{ij}]$.\par
Step 2: We show that if $A$ is not a diagonal matrix and $a_{ij}\not= 0$ then for some $k\not=i$ there  is $E_{ik}\in\mathfrak{I}$.\\
We have
$$[A,E_{ii}]=\sum_{n\in\mathbb{N}}a_{in}E_{in}-\sum_{m\in\mathbb{N}}a_{mi}E_{mi}\in\mathfrak{I}$$ so
$$[\sum_{n\in\mathbb{N}}a_{in}E_{in}-\sum_{m\in\mathbb{N}}a_{mi}E_{mi},E_{jj}]=a_{ij}E_{ij}-a_{ji}E_{ji}\in\mathfrak{I}.$$
If $a_{ji}=0$ we obtain $a_{ij}E_{ij}\in\mathfrak{I}$ (and $E_{ij}\in\mathfrak{I}$, because $\mathfrak{I}$ is the linear space). If $a_{ji}\not= 0$ we have
$$[a_{ij}E_{ij}-a_{ji}E_{ji},E_{ij}]=-a_{ji}E_{jj}+a_{ji}E_{ii}\in\mathfrak{I}$$
so $E_{ii}-E_{jj}\in\mathfrak{I}$. Using step $1$ we obtain $E_{ik}\in\mathfrak{I}$ for all $k\not= i$ if $\textrm{char }\mathbb{K}\not=2$ or
$E_{ik}\in\mathfrak{I}$ for all $k\notin \{i,j\}$ if $\textrm{char }\mathbb{K}=2$.\par
Step 3: We show that if for some $i\not= j$ $E_{ij}\in\mathfrak{I}$ then $\mathfrak{I}$ contains all matrices $B\in\mathfrak{sl}_{fr}$.\\
For every matrix $A\in\mathfrak{gl}_{cf}$ we have
$$[A,E_{ij}]=\sum_{m\in\mathbb{N}}a_{mi}E_{mj}-\sum_{n\in\mathbb{N}}a_{jn}E_{in}=\sum_{m=i\vee n=j}b_{mn}E_{mn},$$
where $b_{ij}=a_{ii}-a_{jj}$, $b_{in}=-a_{jn}$ for $n\notin\{i,j\}$, $b_{mj}=a_{mi}$ for $m\notin\{i,j\}$, $b_{ii}=-a_{ji}$ and $b_{jj}=a_{ji}$. Because $A$ is
an arbitrary matrix we can choose entries $b_{mn}$ in any way, with condition $b_{ii}=-b_{jj}$ (which means that the trace is equal to zero). In particular we obtain
that for all $m\not=j$
$E_{mj}\in\mathfrak{I}$ and for all $n\not=i$ $E_{in}\in\mathfrak{I}$. Repeating the argument from the beginning of the paragraph we obtain
$\sum_{m=i\vee n=j}b_{mn}E_{mn}\in\mathfrak{I}$ for all natural numbers $i\not=j$. Since $\mathfrak{I}$ is a linear space, finite sum of matrices of this type
also is in $\mathfrak{I}$ and the proof is complete.
\end{proof}


\begin{lem}\label{gl}
Let  $\mathfrak{I}\triangleleft \mathfrak{gl}_{cf}$. Then either $\mathfrak{I}\subset \mathfrak{d}_{sc} \oplus \mathfrak{gl}_{fr}$ or $\mathfrak{I}=\mathfrak{gl}_{cf}$.  
\end{lem}
\begin{proof}
Step 1: We show that if  $A\in\mathfrak{I}\setminus(\mathfrak{d}_{sc}\oplus\mathfrak{gl}_{fr})$ is not a diagonal matrix then $\mathfrak{I}$ contains a diagonal matrix $D$ for which the set
$\{i\in\mathbb{N}: d_{ii}\not=d_{i+1,i+1}\}$ is infinite.\\
Assume at first that $A$ has finite number of nonzero entries beyond diagonal. We know that for all natural numbers $i\not=j$ we have $E_{ij}\in\mathfrak{I}$ (lemma
\ref{sl}). Since $\sum_{i\not=j}a_{ij}E_{ij}$ is the finite sum,  $A-\sum_{i\not=j}a_{ij}E_{ij}$ is a diagonal matrix, $A-\sum_{i\not=j}a_{ij}E_{ij}\in\mathfrak{I}$
and $A-\sum_{i\not=j}a_{ij}E_{ij}\notin\mathfrak{gl}_{fr} \oplus \mathfrak{d}_{sc}$.\\
Assume now that $A$ has infinite number of nonzero entries beyond the diagonal. 
Choose a sequence of nonzero entries $a_{i_1,j_1}$, $a_{i_2,j_2}$, $a_{i_3,j_3}$,
\dots in such a way that
for all $n\in\mathbb{N}$ $i_1>1$, $j_1>1$, $i_n\not=j_n$, $i_{n+1}>1+i_n$, $i_{n+1}>1+j_n$, $j_{n+1}>1+j_n$, $j_{n+1}>1+i_n$, $1+k_n<i_{n+1}$ and $1+k_n<j_{n+1}$
where $k_n$ is the  number of row with last nonzero entry in column $j_n$. We can assume also that $i_n>j_n$ for all $n$ or $i_n<j_n$ for all $n$. 
We write the
proof for $i_n<j_n$,  a second case is similar (by transposing the corresponding matrices). We have
$$[\sum_{n\in\mathbb{N}}E_{i_n,i_n},A]=\sum_{n\in\mathbb{N}}\sum_{m\in\mathbb{N}}a_{i_n,m}E_{i_n,m}-\sum_{n\in\mathbb{N}}\sum_{m\in\mathbb{N}}a_{n,i_m}E_{n,i_m}$$
and
$$[\sum_{n\in\mathbb{N}}\sum_{m\in\mathbb{N}}a_{i_n,m}E_{i_n,m}-\sum_{n\in\mathbb{N}}\sum_{m\in\mathbb{N}}a_{n,i_m}E_{n,i_m},\sum_{n\in\mathbb{N}}E_{j_n,j_n}]=$$
$$=\sum_{n\in\mathbb{N}}\sum_{m\in\mathbb{N}}a_{i_n,j_m}E_{i_n,j_m}+\sum_{n\in\mathbb{N}}\sum_{m\in\mathbb{N}}a_{j_n,i_m}E_{j_n,i_m}\in\mathfrak{I}$$
so
\begin{displaymath}\begin{split}
&[\sum_{n\in\mathbb{N}}E_{i,i+1},\sum_{n\in\mathbb{N}}\sum_{m\in\mathbb{N}}a_{i_n,j_m}E_{i_n,j_m}+\sum_{n\in\mathbb{N}}\sum_{m\in\mathbb{N}}a_{j_n,i_m}E_{j_n,i_m}
]=\\
&=\sum_{n\in\mathbb{N}}\sum_{m\in\mathbb{N}}a_{i_n,j_m}E_{i_n-1,j_m}+\sum_{n\in\mathbb{N}}\sum_{m\in\mathbb{N}}a_{j_n,i_m}E_{j_n-1,i_m}-\\
&-\sum_{n\in\mathbb{N}}\sum_{m\in\mathbb{N}}a_{i_n,j_m}E_{i_n,j_m+1}-\sum_{n\in\mathbb{N}}\sum_{m\in\mathbb{N}}a_{j_n,i_m}E_{j_n,i_m+1}\in\mathfrak{I}.
\end{split}\end{displaymath}
We obtain
\begin{displaymath}\begin{split}
&[\sum_{n\in\mathbb{N}}E_{i_n-1,i_n-1},\sum_{n\in\mathbb{N}}\sum_{m\in\mathbb{N}}a_{i_n,j_m}E_{i_n-1,j_m}+\sum_{n\in\mathbb{N}}\sum_{m\in\mathbb{N}}
a_{j_n,i_m}E_{j_n-1,i_m}-\\
&-\sum_{n\in\mathbb{N}}\sum_{m\in\mathbb{N}}a_{i_n,j_m}E_{i_n,j_m+1}-\sum_{n\in\mathbb{N}}\sum_{m\in\mathbb{N}}a_{j_n,i_m}E_{j_n,i_m+1}]=\\
&=\sum_{n\in\mathbb{N}}\sum_{m\in\mathbb{N}}a_{i_n,j_m}E_{i_n-1,j_m}\in\mathfrak{I}.
\end{split}\end{displaymath}

We have 
$$\sum_{n\in\mathbb{N}}\sum_{m\in\mathbb{N}}a_{i_n,j_m}E_{i_n-1,j_m}=\sum_{n\in\mathbb{N}}\sum_{m\geq n}a_{i_n,j_m}E_{i_n-1,j_m}$$
 because $1+k_n<i_{n+1}$.
Denote for simplicity $h_n=i_n-1$, $a_{i_n,j_m}=b_{h_n,j_m}$ so 
$$\sum_{n\in\mathbb{N}}\sum_{m\geq n}a_{i_n,j_m}E_{i_n-1,j_m}=\sum_{n\in\mathbb{N}}
\sum_{m\geq n}b_{h_n,j_m}E_{h_n,j_m}.$$ 
We define
$$x_{h_n,h_m}=\begin{cases}
-b_{h_m,j_m}^{-1}b_{h_n,j_m}&\textrm{for }m=n+1,\\
-b_{h_m,j_m}^{-1}(b_{h_n,j_m}-\sum_{k=n+1}^{m-1}x_{h_n,h_k}b_{h_k,j_m}) &\textrm{otherwise}.
\end{cases} $$
We have
$$[\sum_{n\in\mathbb{N}}\sum_{m>n}x_{h_n,h_m}E_{h_n,h_m},\sum_{n\in\mathbb{N}}\sum_{m\geq n}b_{h_n,j_m}E_{h_n,j_m}]=\sum_{n\in\mathbb{N}}b_{h_n,j_n}E_{h_n,j_n}.$$
Finally
$$[\sum_{n\in\mathbb{N}}b_{h_n,j_n}E_{h_n,j_n},\sum_{n\in\mathbb{N}}E_{j_n,h_n}]=\sum_{n\in\mathbb{N}}b_{h_n,j_n}E_{h_n,h_n}-\sum_{n\in\mathbb{N}}
b_{h_n,j_n}E_{j_n,j_n}\in\mathfrak{I}$$
is the required diagonal matrix.
\par
Step 2: We show that if $A$ is a diagonal matrix and if the set $$\mathbf{H}=\{i\in\mathbb{N}: a_{ii}\not=a_{i+1,i+1}\}$$ is infinite, then
$\sum_{i\in\mathbf{H}}E_{i,i+1}\in\mathfrak{I}$.
We have
\begin{displaymath}\begin{split}& [A,\sum_{i\in\mathbf{H}}(a_{ii}-a_{i+1,i+1})^{-1}E_{i,i+1}]=\\
&=\sum_{i\in\mathbf{H}}a_{ii}(a_{ii}-a_{i+1,i+1})^{-1}E_{i,i+1}-\sum_{i\in\mathbf{H}}a_{i+1,i+1}(a_{ii}-a_{i+1,i+1})^{-1}E_{i,i+1}=\\
&=\sum_{i\in\mathbf{H}}E_{i,i+1}.
\end{split}\end{displaymath}
\par
Step 3: Let $\mathbf{H}$ be an infinite subset of $\mathbb{N}$. We show that if $\sum_{i\in\mathbf{H}}E_{i,i+1}\in\mathfrak{I}$ then
$\sum_{i\in\mathbf{G}}E_{i,i+1}\in\mathfrak{I}$, where $\mathbf{H}\subset\mathbf{G}\subset\mathbb{N}$ and $\mathbb{N}\setminus\mathbf{G}$ does not contain two
consecutive natural numbers.\\
If $\mathbf{H}$ satisfies the assumptions of this step we can take $\mathbf{H}=\mathbf{G}$. Otherwise
define the family of sets $\{\mathbf{Z}_n\}_{n\in\mathbb{N}}$ as $\mathbf{Z}_1=\emptyset$ and
$$\mathbf{Z}_n=\begin{cases}\{n\}\mbox{  }\,\,\textrm{if }\{n-1,n\}\cap\mathbf{H}=\emptyset \wedge \mathbf{Z}_{n-1}=\emptyset,\\
\emptyset\mbox{  } \,\,\,\,\,\,\,\,\,\textrm{otherwise.} \end{cases}$$
Let $\mathbf{Z}=\cup_{n=1}^{\infty}\mathbf{Z}_n$. Note that $\mathbf{Z}\cap\mathbf{H}=\emptyset$ and $\mathbb{N}\setminus (\mathbf{Z}\cup\mathbf{H})$ does not
contain two consecutive natural numbers. Denote $G:=\mathbf{Z}\cup\mathbf{H}$. Let $f:\mathbf{Z}\rightarrow\mathbf{H}$ be a strictly increasing function and
$f(n)>n$ for all $n\in\mathbf{Z}$. We have
$$[\sum_{i\in\mathbf{Z}}E_{i,f(i)},\sum_{j\in\mathbf{H}}E_{j,j+1}]=\sum_{i\in\mathbf{Z}}E_{i,f(i)+1}\in\mathfrak{I}$$
so
$$[\sum_{i\in\mathbf{Z}}E_{i,f(i)+1},\sum_{i\in\mathbf{Z}}E_{f(i)+1,i}]=\sum_{i\in\mathbf{Z}}E_{i,i+1}\in\mathfrak{I}.$$
Finally $$\sum_{i\in\mathbf{H}}E_{i,i+1}+\sum_{i\in\mathbf{Z}}E_{i,i+1}=\sum_{i\in\mathbf{G}}E_{i,i+1}\in\mathfrak{I}.$$\par
Step 4: We show that if $\sum_{i\in\mathbf{G}}E_{i,i+1}\in\mathfrak{I}$ and $\mathbb{N}\setminus\mathbf{G}$ does not contain two
consecutive natural numbers then $\sum_{n\in\mathbb{N}}E_{i,i+1}\in\mathfrak{I}$.\\
We can assume that $1\in\mathbf{G}$ because $E_{12}\in\mathbf{G}$ (from lemma \ref{sl}) and $\mathfrak{I}$ is a linear space. Denote
$\mathbb{N}\setminus\mathbf{G}=\{n_1,n_2,n_3,\dots\}$ and define sets $\mathbf{Z}_1=\{n_1,n_3,n_5,\dots\}$ and $\mathbf{Z}_2=\{n_2,n_4,n_6,\dots\}$. Observe that
for all $m,n\in\mathbf{Z}_1$ or $m,n\in\mathbf{Z}_2$ is satisfied inequality $|m-n|\geq 4$. We have
$$[\sum_{j\in\mathbf{Z}_1}E_{j,j-1},\sum_{i\in\mathbf{G}}E_{i,i+1}]=\sum_{j\in\mathbf{Z}_1}E_{jj}-\sum_{j\in\mathbf{Z}_1}E_{j-1,j-1}\in\mathfrak{I}$$
so
$$[\sum_{j\in\mathbf{Z}_1}E_{jj}-\sum_{j\in\mathbf{Z}_1}E_{j-1,j-1},\sum_{i\in\mathbf{Z}_1}E_{i,i+1}]=\sum_{j\in\mathbf{Z}_1}E_{j,j+1}\in\mathfrak{I}$$
and 
$$\sum_{i\in\mathbf{G}}E_{i,i+1}+\sum_{j\in\mathbf{Z}_1}E_{j,j+1}=\sum_{i\in\mathbf{G}\cup\mathbf{Z}_1}E_{i,i+1}\in\mathfrak{I}.$$ 
Analogously
$$[\sum_{j\in\mathbf{Z}_2}E_{j,j-1},\sum_{i\in\mathbf{G}\cup\mathbf{Z}_1}E_{i,i+1}]=\sum_{j\in\mathbf{Z}_2}E_{jj}-\sum_{j\in\mathbf{Z}_2}E_{j-1,j-1}\in
\mathfrak{I}$$
so
$$[\sum_{j\in\mathbf{Z}_2}E_{jj}-\sum_{j\in\mathbf{Z}_2}E_{j-1,j-1},\sum_{i\in\mathbf{Z}_2}E_{i,i+1}]=\sum_{j\in\mathbf{Z}_2}E_{j,j+1}\in\mathfrak{I}$$
and finally 
$$\sum_{i\in\mathbf{G}\cup\mathbf{Z}_1}E_{i,i+1}+\sum_{j\in\mathbf{Z}_2}E_{j,j+1}=\sum_{n\in\mathbb{N}}E_{n,n+1}\in\mathfrak{I}.$$
\par
Step 5: We show that if $\sum_{i\in\mathbb{N}}E_{i,i+1}\in\mathfrak{I}$ then $\mathfrak{I}=\mathfrak{gl}_{cf}$.\\
We show that equation $[X,\sum_{i\in\mathbb{N}}E_{i,i+1}]=A$ has a solution for all $A\in\mathfrak{gl}_{cf}$. This equation is equivalent to a system of equations:
$$\begin{cases}
x_{m,n-1}-x_{m+1,n}=a_{m,n}\\
x_{m+1,1}=a_{m,1}
\end{cases}\mbox{  } n,m\in\mathbb{N},1<n,$$
which has, for example, the solution
$$
\begin{cases}
x_{1n}=0\mbox{  } n\in\mathbb{N},\\
x_{m+1,1}=a_{m,1}\mbox{  } m\in\mathbb{N},\\
x_{m+1,n}=x_{m,n-1}-a_{mn}\mbox{  } m,n\in\mathbb{N},1<n.
\end{cases}  
$$
\end{proof} 

\noindent {\sc Proof of Theorem 1.1.}

We have $\dim \{0\}=0$ and $\dim\mathfrak{d}_{sc}=1$ so there are no proper ideal in interval $\langle\{0\},\mathfrak{d}_{sc}\rangle$. 
Similarly, since $\mathfrak{sl}_{fr}$ in $\mathfrak{d}_{sc} \oplus \mathfrak{sl}_{fr}$, $\mathfrak{gl}_{fr}$ in $\mathfrak{d}_{sc} \oplus \mathfrak{gl}_{fr}$ and
$\mathfrak{sl}_{fr}$ in $\mathfrak{gl}_{fr}$ have codimension $1$, it follows that intervals
 $\langle\mathfrak{sl}_{fr}, \mathfrak{d}_{sc} \oplus \mathfrak{sl}_{fr}\rangle$, $\langle\mathfrak{gl}_{fr}, \mathfrak{d}_{sc} \oplus \mathfrak{gl}_{fr}\rangle$ and
 $\langle\mathfrak{sl}_{fr}, \mathfrak{gl}_{fr}\rangle$
do not contain proper ideals. From Lemma $2.2$ and $2.3$  it follows that the same holds for intervals 
$\langle\{0\}, \mathfrak{sl}_{fr}\rangle$, $\langle\mathfrak{d}_{sc} \oplus \mathfrak{gl}_{fr}, \mathfrak{gl}_{cf}\rangle$.

As a vector space $\mathfrak{d}_{sc} \oplus \mathfrak{gl}_{fr}$ is generated by 
$E, E_{1,1}$ and $\mathfrak{sl}_{fr}$, so the span of $\alpha E + \beta E_{1,1}$ and $\mathfrak{sl}_{fr}$ gives a codimension $1$ subspace of $\mathfrak{d}_{sc} \oplus \mathfrak{gl}_{fr}$ (for fixed $\alpha, \beta \in K$). It is easy to check that it is an ideal of $\mathfrak{d}_{sc} \oplus \mathfrak{gl}_{cf}$. If $\alpha \not=0$ and $\beta =0$, then it coincides with $\mathfrak{d}_{sc} \oplus \mathfrak{sl}_{fr}$, if  $\alpha =0$ and $\beta \not=0$, then it coincides with $\mathfrak{gl}_{fr}$.

To avoid repetitions we use one-parameter sets $\mathfrak{J}_{k}$. It is clear that $\mathfrak{d}_{sc} \oplus \mathfrak{sl}_{fr}$, $\mathfrak{J}_{k}$ for $k\not=0$ and $\mathfrak{gl}_{fr}$ are all ideals of $\mathfrak{gl}_{cf}$ of codimension $1$ in 
 $\mathfrak{d}_{sc} \oplus \mathfrak{gl}_{cf}$. 
 
The theorem is proved.   \qed

\bigskip

The Corollary $1.2.$ follows from homomorphism theorems for Lie algebras. In [HKZ] it was shown that $\sum_{i=1}^{\infty} E_{i,i+2}$ belongs to the derived Lie subalgebra of $\mathfrak{n}(\mathbb{N}, R)$ and so to the derived Lie subalgebra of $\mathfrak{gl}_{cf}(\mathbb{N},R)$ which shows the first part of Corollary $1.3$. The second part is the obvious consequence of Theorem $1.1.$


\section{Derivations of column-finite infinite matrices}

In this section by $R$ we denote a commutative, associative unital ring. The $R$-linear  map 
$$\varphi:  \mathfrak{gl}_{cf}(\mathbb{N}, R)  \rightarrow  \mathfrak{gl}_{cf}(\mathbb{N}, R)$$
is called a derivation if 
$$\varphi ([X,Y])=[\varphi (X),Y]+[X,\varphi (Y)].$$

The set 
$Der(\mathfrak{gl}_{cf}(\mathbb{N}, R))$ of all derivations of $\mathfrak{gl}_{cf}(\mathbb{N}, R)$ 
forms a Lie algebra.
For a fixed 
$A \in \mathfrak{gl}_{cf}(\mathbb{N}, R)$ 
the map 
$ \ad  A: X \rightarrow [A,X] $ 
is a derivation, called inner derivation. A derivation is called central  if $Im(\varphi)$ is a subset of the center.

By $\mathcal{ND}(\mathbb{N},R)$ we denote a subset  of $\mathfrak{gl}_{cf}(\mathbb{N}, R)$  of all matrices which have only zeros on the  main diagonal and by $\mathfrak{d}=\mathfrak{d}(\mathbb{N}, R)$ the Lie subalgebra of all diagonal matrices. Any matrix $X \in \mathfrak{gl}_{cf}(\mathbb{N}, R)$  is a sum of two matrices $X= D+N$ where $D$ is diagonal and $N \in\mathcal{ ND}(\mathbb{N}, R)$.

\begin{lem}
For each derivation $\varphi \in Der(\mathfrak{gl}_{cf}(\mathbb{N}, R))$ there exists  $A$ in $\mathcal{ND}(\mathbb{N}, R)$ such that $(\varphi -\ad A)(\mathfrak{d})\subseteq  \mathfrak{d}$.
\end{lem}
\begin{proof}
 We denote $\varphi(E_{kk})=[h^{(k)}_{ij}]$. Let $\varphi(D)=C$ where $D\in \mathfrak{d}$ and $C\in \mathfrak{gl}_{cf}(\mathbb{N}, R)$\\
Step 0: We prove that $h^{(k)}_{ij}$ can be different from zero only if $i=j$ or $i=k$ or $j=k$.\\
For any $d\in D$ we have $[d,E_{kk}]=0$ so
\begin{displaymath}\begin{split}
0=\varphi ([d,E_{kk}])&=[d,\varphi (E_{kk})]+[\varphi (d),E_{kk}]=\\
&=d\varphi (E_{kk})-\varphi (E_{kk})d+\varphi (d)E_{kk}-E_{kk}\varphi (d).
\end{split}\end{displaymath}
If $i\neq k$ and $j\neq k$ we have the equality $d_{ii}h^{(k)}_{ij}-h^{(k)}_{ij}d_{jj}=h^{(k)}_{ij}(d_{ii}-d_{jj})=0$. Matrix $d$ is arbitrary,  so we can assume
that $d_{ii}\neq d_{jj}$ if $i\neq j$. In this case we obtain $h^{(k)}_{ij}=0$.\par
Step 1: We prove that $c_{ij}=d_{ii}h^{(i)}_{ij}+d_{jj}h^{(j)}_{ij}$  if $i\not= j$.\\
Let $A=D-d_{ii}E_{ii}-d_{jj}E_{jj}$ and $\varphi (A)=B$.
We have $[E_{ii},A]=0$ so
\begin{displaymath}\begin{split}
0=\varphi ([E_{ii},a])&=[E_{ii},\varphi (a)]+[\varphi (E_{ii}),A]=\\
&=E_{ii}\varphi (A)-\varphi (A)E_{ii}+\varphi (E_{ii})A-A\varphi (E_{ii}).
\end{split}\end{displaymath}
We obtain equality $0=b_{ij}-0+0-0=b_{ij}$.
We have 
$$C=\varphi (D)=\varphi (A+d_{ii}E_{ii}+d_{jj}E_{jj})=\varphi (A)+d_{ii}\varphi (E_{ii})+d_{jj}\varphi (E_{jj})$$ so
$c_{ij}=0+d_{ii}h^{(i)}_{ij}+d_{jj}h^{(j)}_{ij}$.\par
\smallskip
Step 2: We prove that $h^{(i)}_{ij}=-h^{(j)}_{ij}$ for all  $i\neq j$.\\
Since
\begin{displaymath}\begin{split}
0&=\varphi ([E_{ii},E_{jj}])=[E_{ii},\varphi (E_{jj})]+[\varphi (E_{ii}),E_{jj}]=\\
&=E_{ii}\varphi (E_{jj})-\varphi (E_{jj})E_{ii}+\varphi(E_{ii})E_{jj}-E_{jj}\varphi (E_{ii})
\end{split}\end{displaymath}
we get $0=h^{(j)}_{ij}+h^{(i)}_{ij}$.\par
\smallskip
Step 3: Denote $\varphi(D)=C=X+Y$ where $X\in \mathfrak{d}$ and $Y\in \mathcal{ND}(\mathbb{N}, R)$. From the previous steps it follows
$$y_{ij}=d_{ii}h^{(i)}_{ij}+d_{jj}h^{(j)}_{ij}=d_{ii}h^{(i)}_{ij}-d_{jj}h^{(i)}_{ij}.$$
Let $A\in \mathcal{ND}(\mathbb{N}, R)$, $a_{ij}=-h^{(i)}_{ij}$. We have 
$$Y=D(-A)-(-A)D=AD-DA=[A,D],$$
and so $$\varphi(D)=X+Y=X+[A,D]=X+\ad A (D).$$
Finally 
$$\varphi(D)-\ad A (D)=(\varphi-\ad A)(D)=X.$$ \qed
\end{proof}

\begin{lem}
For each derivation $\varphi \in Der(\mathfrak{gl}_{cf}(\mathbb{N}, R))$  there exists  $B\in \mathfrak{gl}_{cf}(\mathbb{N}, R)$ such that $(\varphi -\ad B)(\mathfrak{d})\subseteq \mathfrak{d}$ and $(\varphi -\ad B)(\mathcal{ND}(\mathbb{N}, R))\subseteq \{0\}$.
\end{lem}
\begin{proof}
For each derivation $\varphi$ we have $$(\varphi -\ad A)(\mathfrak{d})\subseteq \mathfrak{d}$$ for some $A\in \mathcal{ND}(\mathbb{N}, R)$. Denote $\psi =\varphi -\ad A$.\par
\smallskip
Step 1: We will prove that $\psi (E_{i, i+1})=x_iE_{i, i+1}$ where $x_i\in R$.\\
For any $D\in \mathfrak{d}$ we have:
\begin{displaymath}\begin{split}
\psi ([D,E_{i, i+1}])&=\psi (DE_{i,i+1}-E_{i, i+1}D)=\\
&=\psi (d_{ii}E_{i, i+1}-d_{i+1,i+1}E_{i, i+1})=\\
&=(d_{ii}-d_{i+1,i+1})\psi (E_{i, i+1}).
\end{split}\end{displaymath}
On the other hand 
\begin{displaymath}\begin{split}
\psi ([D,E_{i, i+1}])&=[D,\psi (E_{i, i+1})]+[\psi (D),E_{i, i+1}]=\\
&=D\psi (E_{i, i+1})-\psi (E_{i, i+1})D+\psi (D)E_{i, i+1}-E_{i, i+1}\psi (D).
\end{split}\end{displaymath}
 Denote $\psi (E_{i, i+1})=C$. $\psi (D)$ is a diagonal matrix so $$\psi (D)E_{i, i+1}-E_{i, i+1}\psi (D)=rE_{i, i+1}$$ where $r\in R$.
For $(j,k)\neq (i,i+1)$ we obtain equality $$(d_{ii}-d_{i+1, i+1})c_{jk}=(d_{jj}-d_{kk})c_{jk}.$$ We can assume that $d_{ii}-d_{i+1,i+1}=0$ and $d_{jj}-d_{kk}=1$, then we have $c_{jk}=0$
for $j\not=k$. On the other hand if we take $D=E_{ii}$ we obtain $c_{jj}=0$ for all $j\in\mathbb{N}$.
\par
\smallskip
Step 2: We will prove that $\psi (E_{i j})=(x_i+x_{i+1}+...+x_{j-1})E_{i j}$ for $i<j$.\\
Denote $j=i+k$. We will use mathematical induction on $k$. From the previous step we get equality when $k=1$. We assume that this property is satisfied for some $k$.
For $k+1$ we have 
\begin{displaymath}\begin{split}
\psi (E_{i,i+k+1})&=\psi ([E_{i,i+1},E_{i+1,i+k+1}])=\\
&=[E_{i,i+1},\psi (E_{i+1,i+k+1})]+[\psi (E_{i,i+1}),E_{i+1,i+k+1}]=\\
&=E_{i,i+1}\psi (E_{i+1,i+k+1})-\psi (E_{i+1,i+k+1})E_{i,i+1}+\\
&-\psi (E_{i,i+1})E_{i+1,i+k+1}-E_{i+1,i+k+1}\psi (E_{i,i+1})=\\
&=(x_{i+1}+...+x_{i+k}))E_{i,i+1}E_{i+1,i+k+1}-\\
&-(x_{i+1}+...+x_{i+k}))E_{i+1,i+k+1}E_{i,i+1}+\\
&+x_iE_{i,i+1}E_{i+1,i+k+1}-x_iE_{i+1,i+k+1}E_{i,i+1}=\\
&=(x_{i+1}+...+x_{i+k}))E_{i,i+k+1}-0+x_iE_{i,i+k+1}-0=\\
&=(x_i+x_{i+1}+...+x_{i+k})E_{i,i+k+1}.
\end{split}\end{displaymath}
\smallskip
Step 3: We will prove that $\psi (D)= r_D E$ for all $D\in \mathfrak{d}$ and $r_D\in R$.
Let  $\psi (D)=C$. For $i\in\mathbb{N}$ we have
$$\psi ([D,E_{i,i+1}])=\psi ((d_{ii}-d_{i+1,i+1})E_{i,i+1})=(d_{ii}-d_{i+1,i+1})x_i E_{i,i+1}.$$
On the other hand
\begin{displaymath}\begin{split}
\psi ([D,E_{i,i+1}])&=[D,\psi (E_{i,i+1})]+[\psi (D),E_{i,i+1}]=\\
&=D\psi (E_{i,i+1})-\psi (E_{i,i+1})D+\psi (D)E_{i,i+1}-E_{i,i+1}\psi (D)=\\
&=d_{ii}x_i E_{i,i+1}-d_{i+1,i+1}x_i E_{i,i+1}+c_{ii}E_{i,i+1}-c_{i+1,i+1}E_{i,i+1}=\\
&=(d_{ii}-d_{i+1,i+1})x_iE_{ii+1}+(c_{ii}-c_{i+1i+1})E_{ii+1}.
\end{split}\end{displaymath}
We get $c_{ii}-c_{i+1,i+1}=0$ for $i\in\mathbb{N}$, so $\psi (D)=r_D E$ where $r_D\in R$.\par
\smallskip
Step 4: We will prove that $\psi (E_{i j})=-\psi (E_{j i})$ for $i\not= j$.\\
Analogously to steps 1 and 2, we can show that $\psi (E_{j i})=rE_{ji}$ for $j>i$ and some $r\in R$. Denote $\psi (E_{i j})=x_{ij}E_{ij}$ for
$i\not= j$. We have 
$$\psi ([E_{ij},E_{ji}])=\psi (E_{ii}-E_{jj})=\psi (E_{ii})-\psi(E_{jj})=r_iE-r_jE.$$
On the other hand
\begin{displaymath}\begin{split}
\psi ([E_{ij},E_{ji}])&=\psi (E_{ij})E_{ji}-E_{ji}\psi (E_{ij})+E_{ij}\psi(E_{ji})-\psi(E_{ji})E_{ij}=\\
&=x_{ij}E_{ii}-x_{ij}E_{jj}+x_{ji}E_{ii}-x_{ji}E_{jj}.
\end{split}\end{displaymath}
In this way we obtain that  $r_i=r_j$ and $$x_{ij}E_{ii}-x_{ij}E_{jj}+x_{ji}E_{ii}-x_{ji}E_{jj}=0.$$ So $x_{ij}+x_{ji}=0$ for $i\not= j$.
\par
\smallskip
Step 5: Let $j>i+5$ and $\beta ^{(i,j)}$ be a set of matrices which can have nonzero entries only in the $i$-th row  in the columns $k\geq j$. We will prove that
$\psi (\beta ^{(i,j)})\subset \beta ^{(i,j)}$.\\
Let $B^{(i,j)}\in \beta ^{(i,j)}$. We have
\begin{displaymath}\begin{split}
\psi(B^{(i,j)})&=\psi([E_{ii},B^{(i,j)}])=[\psi(E_{ii}),B^{(i,j)}]+[E_{ii},\psi(B^{(i,j)})]=\\
&=0+[E_{ii},\psi(B^{(i,j)})]=E_{ii}\psi(B^{(i,j)})-\psi(B^{(i,j)})E_{ii},
\end{split}\end{displaymath}
so $\psi(B^{(i,j)})$ can have nonzero entries only in the $i$-th row and the  $i$-th column. For $k<j$ and $k\not=i-1$ we have
$$\psi([B^{(i,j)},E_{k,k+1}])=\psi(B^{(i,j)}E_{k,k+1}-E_{k,k+1}B^{(i,j)})=\psi(0-0)=0.$$
On the other hand
\begin{displaymath}\begin{split}
&\psi([B^{(i,j)},E_{k,k+1}])=[\psi(B^{(i,j)}),E_{k,k+1}]+[B^{(i,j)},\psi(E_{k,k+1})]=\\
&=\psi(B^{(i,j)})E_{k,k+1}-E_{k,k+1}\psi(B^{(i,j)})+\\
&+x_{k,k+1}(B^{(i,j)}E_{k,k+1}-E_{k,k+1}B^{(i,j)})=\\
&=\psi(B^{(i,j)})E_{k,k+1}-E_{k,k+1}\psi(B^{(i,j)}),
\end{split}\end{displaymath}
so $k$-th column  of the matrix $\psi(B^{(i,j)})$ have only $0$'s in every row. We get $\psi (\beta ^{(1,j)})\subseteq \beta ^{(1,j)}$ and it is enough to show that
$\psi(B^{(i,j)})_{i,i-1}=0$ for $i>1$. We shall use induction to end the proof. We have
$$\psi([B^{(i,j)},E_{i-1,i}])=\psi(B^{(i,j)}E_{i-1,i}-E_{i-1,i}B^{(i,j)})=\psi(0-E_{i-1,i}B^{(i,j)})\subset\beta ^{(i-1,j)}$$
and
\begin{displaymath}\begin{split}
\psi([B^{(i,j)},E_{i-1,i}])=\psi(B^{(i,j)})E_{i-1,i}-E_{i-1,i}\psi(B^{(i,j)})+\\
+x_{i-1,i}(B^{(i,j)}E_{i-1,i}-E_{i-1,i}B^{(i,j)})=\\
\psi(B^{(i,j)})E_{i-1,i}-E_{i-1,i}\psi(B^{(i,j)})+x_{i-1,i}(0-E_{i-1,i}B^{(i,j)}).
\end{split}\end{displaymath}
Since $$-E_{i-1,i}\psi(B^{(i,j)})+x_{i-1,i}(-E_{i-1,i}B^{(i,j)})\subseteq \beta ^{(i-1,j)}$$ 
we obtain $\psi(B^{(i,j)})E_{i-1,i}=0$ and hence $\psi(B^{(i,j)})_{i,i-1}=0$.
\par
\smallskip
Step 6: Let $N\in \mathcal{ND}(\mathbb{N}, R)$ and $\psi(N)=M$. We will prove that if $n_{ij}=0$ then $m_{ij}=0$.\\
Denote $Y=\psi ([E_{ii},N])$. We have $Y=\psi ([E_{ii},N])=\psi (E_{ii}N-NE_{ii})$. Using step 5 and step 2 we get $$y_{ij}=n_{ij}(x_i+x_{i+1}+...+x_{j-1})=0$$ for $i<j$ or $$y_{ij}=-n_{ij}(x_i+x_{i+1}+...+x_{j-1})=0$$ for $i>j$. On the other hand
\begin{displaymath}\begin{split}
\psi ([E_{ii},N])&=[E_{ii},\psi (N)]+[\psi (E_{ii}),N]=\\&
E_{ii}\psi (N)-\psi (N)E_{ii}+\psi (E_{ii})N-N\psi (E_{ii}).
\end{split}\end{displaymath}
But $\psi (E_{ii})\in \mathfrak{d}$,  hence we have $y_{ij}=m_{ij}-0+0-0=m_{ij}$.\par
\smallskip
Step 7: Let $Y\in \mathcal{ND}(\mathbb{N}, R)$ and $\psi (Y)=Z$. For $i<j$ we have $z_{ij}=y_{ij}(x_i+x_{i+1}+...+x_{j-1})$ and $z_{ji}=-y_{ji}(x_i+x_{i+1}+...+x_{j-1})$.\\
Let $D\in \mathfrak{d}$ and $d_{11}=0$, $d_{ii}=-(x_1+x_2+...+x_{i-1})$ for $i>1$. We have 
\begin{displaymath}\begin{split}
z_{ij}&=y_{ij}(x_i+x_{i+1}+...+x_{j-1})=\\
&=y_{ij}(x_1+x_2+...+x_{j-1})-y_{ij}(x_1+x_2+...+x_{i-1})=\\
&=-d_{jj}y_{ij}+d_{ii}y_{ij}=d_{ii}y_{ij}-d_{jj}y_{ij}
\end{split}\end{displaymath}
 and 
\begin{displaymath}\begin{split}
z_{ji}&=-y_{ji}(x_i+x_{i+1}+...+x_{j-1})=\\
&=-y_{ji}(x_1+x_2+...+x_{j-1})+y_{ij}(x_1+x_2+...+x_{i-1})=\\
&=d_{jj}y_{ji}-d_{ii}y_{ji}
\end{split}\end{displaymath}
 so $\psi (Y)=Z=[D,Y]$. 

This means that for all
$N\in \mathcal{ND}(\mathbb{N}, R)$ we have $\varphi (N)-[A,N]=[D,N]$ and  so $$\varphi (N)-[A,N]-[D,N]=\varphi (N)-[A+D,N]=(\varphi -ad B)(N)=0$$ where $B=A+D$. 
\end{proof}

\begin{lem}
If $\varphi$ is a derivation of $\mathfrak{gl}_{cf}(\mathbb{N}, R)$ such that   $\varphi(D)\subseteq D$ and $\varphi(\mathcal{ND}(\mathbb{N}, R))=\{0\}$ then $\varphi$ is a central derivation.
\end{lem}
\begin{proof}
Let  $\varphi (D)=C$. For $i\in\mathbb{N}$ we have
$$\varphi ([D,E_{ii+1}])=\varphi ((d_{ii}-d_{i+1,i+1})E_{ii+1})=0.$$
On the other hand
\begin{displaymath}\begin{split}
\varphi ([D,E_{i,i+1}])&=[D,\varphi (E_{i,i+1})]+[\varphi (D),E_{i,i+1}]=\\
&=D\varphi (E_{i,i+1})-\varphi (E_{i,i+1})D+\varphi (D)E_{i,i+1}-E_{i,i+1}\varphi (D)=\\
&=0-0+c_{ii}E_{i,i+1}-c_{i+1,i+1}E_{i,i+1}=(c_{ii}-c_{i+1,i+1})E_{i,i+1}.
\end{split}\end{displaymath}
We get $c_{ii}-c_{i+1,i+1}=0$ for $i\in\mathbb{N}$, so $\varphi (D)=r_D E$ where $r_D\in R$.\par
We have $\varphi (D)=r_D E$ where $r_D\in R$. This means that there  exists a homomorphism $\sigma : \mathfrak{d}\rightarrow R$ of
$R$-modules, defined by $\sigma (D)=r_D$. We have 
$$\varphi(A)=\varphi (D+N)=\varphi (D)=\sigma (D)E$$
 for $A\in \mathfrak{gl}_{cf}(\mathbb{N}, R)$, $D\in \mathfrak{d}$, $N\in \mathcal{ND}(\mathbb{N}, R)$. Hence $\varphi $ is the
central derivation of 
$\mathfrak{gl}_{cf}(\mathbb{N},R)$ 
induced by $\sigma $.
\end{proof}

\noindent {\sc Proof of Theorem 1.5.}

By previous Lemmas for each derivation $\varphi $ there exists $B \in \mathfrak{gl}_{cf}(\mathbb{N},R)$ such that $\varphi - \ad B=\hat{\varphi}$, 
where $\hat{\varphi} $ is central derivation, so $\varphi = \ad B+\hat{\varphi} $.
\qed

\begin{lem}
If $\varphi$ is a central derivation of a Lie algebra $L$, then $\varphi([L, L])= \{0\}$.
\end{lem}
\begin{proof}
From a definition of central derivation for all $x, y \in L$ we have $\varphi([x,y])= [\varphi(x), y] + [x, \varphi(y)]= 0+0=0$. By linearity $\varphi({L,L})= \{0\}$.
\end{proof}

\noindent {\sc Proof of Corollary 1.6.}

From Corollary $1.3.$, Theorem $1.5.$ and Lemma $3.4.$ it follows that every derivation of $\mathfrak{gl}_{cf}(\mathbb{N},K)$
is inner. It is easy to check that inner derivation $x \mapsto [y,x]$ is trivial in $\mathfrak{gl}_{cf}(\mathbb{N},K)$
if and only if $y$ belongs to the center. This implies isomorphism in thesis. \qed

\bigskip

Now we prove 
\begin{prop}
Lie algebras $\mathfrak{gl}_{cf}(\mathbb{N},R)$ i $\mathfrak{gl}_{cf}(\mathbb{Z},R)$ are isomorphic.
\end{prop}
\begin{proof}
We recall that Lie algebra $\mathfrak{gl}_{cf}(\mathbb{Z},R)$ consists of column-finite infinite $\mathbb{Z} \times \mathbb{Z}$ matrices.

We define a function $\sigma :\mathbb{Z}\rightarrow \mathbb{N}$ as follows
$$\sigma (x)=\begin{cases} 2x+1\textrm{ dla } x\geq 0\\ -2x\textrm{\,\,\,\, \,\,dla } x< 0\end{cases}.$$
$\sigma$ is the bijection and induces the  bijection  $\phi : M_{cf}(\mathbb{Z},K)\rightarrow M_{cf}(\mathbb{N},K)$ given by
$$\phi (\sum\limits_{i,j\in\mathbb{Z}}a_{ij}E_{ij})=\sum\limits_{i,j\in\mathbb{Z}}a_{ij}E_{\sigma (i)\sigma (j)}.$$
We have
$$\phi(cA)=\sum\limits_{i,j\in\mathbb{Z}}ca_{ij}E_{\sigma (i)\sigma (j)}=c\sum\limits_{i,j\in\mathbb{Z}}a_{ij}E_{\sigma (i)\sigma (j)}=cf(A),$$
\begin{align*}
\phi(A+B)&=\sum\limits_{i,j\in\mathbb{Z}}(a_{ij}+b_{ij})E_{\sigma (i)\sigma (j)}=\\
&=\sum\limits_{i,j\in\mathbb{Z}}a_{ij}E_{\sigma (i)\sigma (j)}+\sum\limits_{i,j\in\mathbb{Z}}b_{ij}E_{\sigma (i)\sigma (j)}=\phi(A)+\phi(B),
\end{align*}
\begin{align*}
\phi(AB)&=\phi (\sum\limits_{i,j\in\mathbb{Z}}(\sum\limits_{n=-\infty}^\infty a_{in}b_{nj})E_{ij})=
\sum\limits_{i,j\in\mathbb{Z}}(\sum\limits_{n=-\infty}^\infty a_{in}b_{nj})E_{\sigma (i)\sigma (j)}=\\
&=\sum\limits_{k,m\in\mathbb{N}}(\sum\limits_{n=-\infty}^\infty a_{\sigma ^{-1}(k)n}b_{n\sigma ^{-1}(m)})E_{km}=\\
&=(\sum\limits_{k,m\in\mathbb{N}}a_{\sigma ^{-1}(k)\sigma ^{-1}(m)}E_{km})(\sum\limits_{k,m\in\mathbb{N}}b_{\sigma ^{-1}(k)\sigma ^{-1}(m)}E_{km})=\\
&=(\sum\limits_{i,j\in\mathbb{Z}}a_{ij}E_{\sigma (i)\sigma (j)})(\sum\limits_{i,j\in\mathbb{Z}}b_{ij}E_{\sigma (i)\sigma (j)})=\phi(A)\phi(B) ,
\end{align*}
and so
$$\phi([A,B])=[\phi(A),\phi(B)].$$
It means that  $\phi $ defines the isomorphism of Lie algebras $\mathfrak{gl}_{cf}(\mathbb{N},R)$ and $\mathfrak{gl}_{cf}(\mathbb{Z},R)$. 
\end{proof}


\end{document}